\documentclass[12pt]{amsart}
\usepackage[latin1]{inputenc}
\usepackage{color}
\usepackage{enumerate}
\usepackage{amssymb}
\usepackage[all]{xy}
\usepackage{hyperref}
\usepackage{verbatim} 
\usepackage{soul} 
\usepackage[normalem]{ulem} 
\usepackage{cancel}

\newtheorem{thm}{Theorem}[section]

\newtheorem{lem}[thm]{Lemma}
\newtheorem{prop}[thm]{Proposition}
\theoremstyle{definition}
\newtheorem{defn}[thm]{Definition}
\theoremstyle{remark}

\numberwithin{equation}{section}
\newcommand{\norm}[1]{\left\Vert#1\right\Vert}
\newcommand{\abs}[1]{\left\vert#1\right\vert}
\newcommand{\set}[1]{\left\{#1\right\}}

\newcommand{\To}{\longrightarrow}

\usepackage{color}
\usepackage{enumerate}
\begin{document}

\setcounter{tocdepth}{1}


\title[Projectivity of the free Banach lattice generated by a lattice]{Projectivity of the free Banach lattice generated by a lattice}
\author[A.\ Avil\'es]{Antonio Avil\'es}
\address{Universidad de Murcia, Departamento de Matem\'{a}ticas, Campus de Espinardo 30100 Murcia, Spain.}
\email{avileslo@um.es}

\author[J.D. Rodr\'iguez Abell\'an]{Jos\'e David Rodr\'iguez Abell\'an}
\address{Universidad de Murcia, Departamento de Matem\'{a}ticas, Campus de Espinardo 30100 Murcia, Spain.}
\email{josedavid.rodriguez@um.es}

\thanks{Authors supported by project MTM2017-86182-P (Government of Spain, AEI/FEDER, EU) and project 20797/PI/18 by Fundaci\'{o}n S\'{e}neca, ACyT Regi\'{o}n de Murcia. Second author supported by FPI contract of Fundaci\'on S\'eneca, ACyT Regi\'{o}n de Murcia.}



\begin{abstract}
We study the projectivity of the free Banach lattice generated by a lattice $\mathbb{L}$ in two cases: when the lattice is finite, and when the lattice is an infinite linearly ordered set. We prove that in the first case it is projective while in the second case \textcolor{red}{\textst{it is not.}}\textcolor{blue}{, if the linear order contains either an increasing sequence without upper bounds or a decreasing sequence without lower bounds, then it is not projective.}
\end{abstract}

\maketitle

\setlength{\parskip}{4mm}

\textcolor{blue}{\textit{We found a mistake in this paper. At some point in the proof of the second part of Lemma 4.2 we considered functions $\varphi_i$ that had to be continuous but they were not. So Lemma 4.2 and Theorem 4.4 must be stated in a weaker form. We indicate how to fix the error: Red text indicates what should be removed and blue text what should be added to the published version in Archiv der Mathematik 113 (2019), 515--524}.}

\section{Introduction}

Free and projective Banach lattices were introduced in \cite{dPW15}. The free Banach lattice $FBL(A)$ generated by a set $A$ is a Banach lattice characterized by the property that every bounded map $T:A\To X$ into a Banach lattice $X$ extends to a unique Banach lattice homomorphism $\hat{T}:FBL(A)\To X$ with the same norm. This idea was generalized in \cite{ART18} and \cite{ARA18}, where the free Banach lattice generated by a Banach space $E$ and by a lattice $\mathbb{L}$ are respectively studied. By a lattice we mean here a set $\mathbb{L}$ together with two operations $\wedge$ and $\vee$ that are the infimum and supremum of some partial order relation on $\mathbb{L}$, and a lattice homomorphism is a function between lattices that commutes with those two operations.   

\begin{defn} \label{FBLGBL} Given a lattice $\mathbb{L}$, the \textit{free Banach lattice generated by $\mathbb{L}$} is a Banach lattice $F$ together with a lattice homomorphism $\phi: \mathbb{L} \longrightarrow F$ such that for every Banach lattice $X$ and every bounded lattice homomorphism $T: \mathbb{L} \longrightarrow X$, there exists a unique Banach lattice homomorphism $\hat{T}: F \longrightarrow X$ such that $|| \hat{T} || = || T ||$ and $T = \hat{T} \circ \phi$.	
\end{defn}

Here, the norm of $T$ is $\norm{T} := \sup \set{\norm{T(x)}_X : x \in \mathbb{L}}$, while the norm of $\hat{T}$ is the usual norm of an operator acting between Banach spaces. This definition determines a Banach lattice that we denote by $FBL\langle \mathbb{L}\rangle$ in an essentially unique way. When $\mathbb{L}$ is a distributive lattice the function $\phi$ is injective and we can view $FBL\langle \mathbb{L} \rangle$ as a Banach lattice which contains a subset lattice-isomorphic to $\mathbb{L}$ in a way that its elements work as free generators modulo the lattice relations on $\mathbb{L}$, cf. \cite{ARA18}. To see that, it is well known that a lattice $\mathbb{L}$ is distributive if, and only if, $\mathbb{L}$ is lattice-isomorphic to a bounded subset of a Banach lattice. Thus, it is clear that if $\phi$ is inyective then $\mathbb{L}$ is distributive. On the other hand, if $\mathbb{L}$ is distributive, we have a bounded injective lattice homomorphism $T : \mathbb{L} \longrightarrow X$ for some Banach lattice $X$. Using the definition of being the free Banach lattice generated by the lattice $\mathbb{L}$, there is $\hat{T}$ such that $\hat{T} \circ \phi = T$. Since $T$ is inyective, $\phi$ is also inyective.

The notions of free and projective objects are closely related in the general theory of categories. In the context of Banach lattices, de Pagter and Wickstead \cite{dPW15} introduced projectivity in the following form:

\begin{defn}\label{projdef} A Banach lattice $P$ is \textit{projective} if whenever $X$ is a Banach lattice, $J$ a closed ideal in $X$ and $Q : X \longrightarrow X/J$ the quotient map, then for every Banach lattice homomorphism $T : P \longrightarrow X/J$ and $\varepsilon > 0$, there is a Banach lattice homomorphism $\hat{T} : P \longrightarrow X$ such that $T = Q \circ \hat{T}$ and $\|\hat{T}\| \leq (1 + \varepsilon)\norm{T}$.
\end{defn}

Some examples of projective Banach lattices given in \cite{dPW15} include $FBL(A)$, $\ell_1$, all finite dimensional Banach lattices and Banach lattices of the form $C(K)$, where $K$ is a compact neighborhood retract of $\mathbb{R}^n$. But we are still far from understanding what the projective Banach lattices are. Such basic questions as whether $c_0$, $\ell_2$ or $C([0,1]^\mathbb{N})$ are projective were left open in \cite{dPW15}.

Since the canonical projective Banach lattice is the free Banach lattice $FBL(A)$, it is natural to think that its variants $FBL[E]$ (the free Banach lattice generated by the Banach space $E$) and $FBL\langle \mathbb{L}\rangle$ may also be projective at least in some cases. In this paper we focus on the case of $FBL\langle \mathbb{L}\rangle$. We prove that $FBL\langle \mathbb{L}\rangle$ is projective whenever $\mathbb{L}$ is a finite lattice, while it is not projective when $\mathbb{L}$ is an infinite linearly ordered set \textcolor{blue}{containining either an increasing sequence without upper bounds or a decreasing sequence without lower bounds}.

If $\mathbb{L}$ is a finite lattice, $FBL\langle \mathbb{L}\rangle$ is a  renorming of a Banach lattice of continuous functions $C(K)$ on a compact neighborhood retract $K$ of $\mathbb{R}^n$, which is projective \cite{dPW15}. Projectivity, however, is not preserved under renorming, because of the $(1+\varepsilon)$ bound required in Definition~\ref{projdef}. Getting this bound will be the key point in the proof.

In the infinite case, we considered only linearly ordered sets, as they are easier to handle than general lattices. We do not know if there is some infinite lattice $\mathbb{L}$ such that $FBL\langle \mathbb{L}\rangle$ is projective.

\section{Preliminaries}

\subsection{Absolute neighborhood retracts.}

An absolute neighborhood retract (ANR)  is a topological space $X$ with the property that whenever $X$ is a subspace of $Y$, then there is an open subset $V$ of $Y$ such that $X\subset V \subset Y$ and $X$ is a retract of $V$, meaning that there is a continuous function $r:V\To X$ such that $r(x)=x$  for all $x\in X$.

The following are two basic facts of the theory that can be found in \cite{vMill} as Theorems 1.5.1 and 1.5.9:

\begin{itemize}
	\item Every closed convex subset of $\mathbb{R}^n$ is ANR.
	\item If $X_1$, $X_2$ are closed subsets of $X$, and $X_1$, $X_2$ and $X_1\cap X_2$ are ANR, then $X_1\cup X_2$ is also ANR.
\end{itemize}

From the two facts above, one can easily prove that every finite union of closed convex subsets of $\mathbb{R}^n$ is ANR, by induction on the number of convex sets in that union.

\subsection{Free Banach lattices.}


We collect the necessary facts and definitions about free Banach lattices from \cite{dPW15,ART18,ARA18}. 

An explicit construction of the free Banach lattice $FBL(A)$ generated by a set $A$ is as follows. For $x \in A$, let $\delta_x:[-1,1]^A \longrightarrow [-1,1]$ be the evaluation function given by $\delta_x(x^*) = x^*(x)$ for every $x^* \in [-1,1]^A$, and for $f:[-1,1]^A \longrightarrow \mathbb{R}$ we define $$\|f\| = \sup \set{\sum_{i = 1}^n \abs{ r_if(x_{i}^{\ast})} :  \text{}r_i \in \mathbb{R}, \text{ } x_i^{\ast} \in [-1,1]^A, \text{ }\sup_{x \in A} \sum_{i=1}^n \abs{r_ix_i^{\ast}(x)} \leq 1 },$$ which we will denote by $\| f \|$ or $\| f \|_{FBL(A)}$. The Banach lattice $FBL(A)$ is the Banach lattice generated by the evaluation functions $\delta_x$ inside the Banach lattice of all functions $f:[-1,1]^A\To\mathbb{R}$ with finite norm. The natural identification of $A$ inside $FBL(A)$ is given by the map $u: A \longrightarrow FBL(A)$ where $u(x) = \delta_x$. Since every function in $FBL(A)$ is an uniform limit of such functions, they are all continuous and positively homogeneous (they commute with multiplication by positive scalars). When $A$ is finite, then $FBL(A)$ consists of \emph{all} continuous and positively homogeneous functions on $[-1,1]^A$, or equivalently in this case, all positively homogeneous functions on $[-1,1]^A$ that are continuous on the boundary $\partial [-1,1]^A$. Thus, when $A$ is finite, $FBL(A)$ is a renorming of the Banach lattice of continuous functions on $\partial [-1,1]^A$.

We can describe of $FBL \langle \mathbb{L} \rangle$ as the quotient of $FBL(\mathbb{L})$ (the free Banach lattice generated by the underlying set of the lattice $\mathbb{L}$) by the closed ideal $\mathcal{I}$ of $FBL(\mathbb{L})$ generated by the set $$\set{\delta_x \vee \delta_y - \delta_{x \vee y},\ \ \delta_x \wedge \delta_y - \delta_{x \wedge y}\ : \ x,y \in \mathbb{L}}.$$ In \cite{ARA18} we prove that, $FBL(\mathbb{L})/\mathcal{I}$, together with the map $\phi: \mathbb{L} \longrightarrow FBL(\mathbb{L})/\mathcal{I}$ given by $\phi(x) = \delta_x + \mathcal{I}$ is the free Banach lattice generated by the lattice $\mathbb{L}$. 

Also in \cite{ARA18} there is a different description of $FBL\langle \mathbb{L} \rangle$ as a space of functions. The construction is analogous to that of $FBL(A)$ but taking into account the lattice structure. Namely, if we see $[-1,1]$ as a lattice, define $$\mathbb{L}^{\ast} = \set {x^{\ast}: \mathbb{L} \longrightarrow [-1,1] : x^{\ast} \text{ is a lattice-homomorphism}}.$$ For every $x \in \mathbb{L}$ consider the evaluation function $\dot{\delta}_x : \mathbb{L}^{\ast} \longrightarrow [-1,1]$ given by $\dot{\delta}_x(x^{\ast}) = x^{\ast}(x)$, and for $f \in \mathbb{R}^{\mathbb{L}^{\ast}}$, define $$ \norm{f}_\ast = \sup \set{\sum_{i = 1}^n \abs{ r_if(x_{i}^{\ast})} :  \text{ }r_i \in \mathbb{R}, \text{ } x_i^{\ast} \in \mathbb{L}^{\ast}, \text{ }\sup_{x \in \mathbb{L}} \sum_{i=1}^n \abs{r_ix_i^{\ast}(x)} \leq 1 }.$$

Let $FBL_*\langle \mathbb{L} \rangle$ be the Banach lattice generated by the evaluations $\set{\dot{\delta}_x : x \in \mathbb{L}}$ inside the Banach lattice of all functions $f \in \mathbb{R}^{\mathbb{L}^{\ast}}$ with $\|f\|_\ast<\infty$, endowed with the norm $\|\cdot\|_\ast$ and the pointwise operations. This, together with the assignment $\phi(x)=\dot{\delta}_x$ is the free Banach lattice generated by $\mathbb{L}$.

Thus, we have two alternative constructions of the free Banach lattice generated by $\mathbb{L}$ that we are denoting as $FBL\langle \mathbb{L} \rangle = FBL(\mathbb{L})/\mathcal{I}$ and $FBL_*\langle \mathbb{L} \rangle$, respectively. There is a natural Banach lattice homomorphism $R:FBL(\mathbb{L})\To FBL_*\langle \mathbb{L} \rangle$ given by restriction $R(f) = f|_{\mathbb{L}^*}$. This is surjective and its kernel is the ideal $\mathcal{I}$, and thus $R$ induces the canonical isometric Banach lattice isomorphism between $FBL\langle \mathbb{L} \rangle$ and $FBL_*\langle \mathbb{L} \rangle$.

\subsection{Projective Banach lattices.} 

We state here a variation of \cite[Theorem 10.3]{dPW15}:

\begin{prop}\label{quotientofprojective}
	Let $P$ be a projective Banach lattice, $\mathcal{I}$ an ideal of $P$ and $\pi:P\To P/\mathcal{I}$ the quotient map. The quotient $P/\mathcal{I}$ is projective if and only if for every $\varepsilon>0$ there exists a Banach lattice homomorphism $u_\varepsilon:P/\mathcal{I}\To P$ such that $\pi\circ u_\varepsilon = id_{P/\mathcal{I}}$ and $\|u_\varepsilon\|\leq 1+\varepsilon$.
\end{prop}

\begin{proof}
	If $P/\mathcal{I}$ is projective, then we can just apply Definition~\ref{projdef}. On the other hand, if we have the above property and we want to check Definition~\ref{projdef}, take $\varepsilon_0>0$, a quotient map $Q:X\mapsto X/J$ and a Banach lattice homomorphism $T:P/\mathcal{I}\To X/J$. Take $\varepsilon$ with $(1+\varepsilon)^2\leq 1+\varepsilon_0$. Since $P$ is projective we can find $S:P\To X$ with $Q\circ S = T\circ\pi$ and $\|S\|\leq (1+\varepsilon)\|T\circ\pi\| = (1+
	\varepsilon)\|T\|$. If we take $\hat{T}= S\circ u_\varepsilon$, then $Q\circ \hat{T} = Q\circ S \circ u_\varepsilon = T\circ \pi\circ u_\varepsilon = T$ and $\|\hat{T}\| \leq (1+\varepsilon)^2\|T\| \leq (1+\varepsilon_0)\|T\|$ as desired. 
\end{proof}

Since $FBL(\mathbb{L})$ is projective \cite[Proposition 10.2]{dPW15}, and the restriction map described above $R:FBL(\mathbb{L})\To FBL_\ast\langle \mathbb{L}\rangle$ is a quotient map \cite{ARA18}, we get, as a particular instance of Proposition~\ref{quotientofprojective},

\begin{prop}\label{caractProyectividad}
	
	Let $\mathbb{L}$ be a lattice and let $R: FBL(\mathbb{L}) \longrightarrow FBL_*\langle \mathbb{L} \rangle$ be the restriction map $R(f) = f\vert_{\mathbb{L}^*}$.	The Banach lattice $FBL_*\langle \mathbb{L} \rangle$ is projective if, and only if, for every $\varepsilon > 0$ there exists a Banach lattice homomorphism $u_{\varepsilon}:  FBL_*\langle \mathbb{L} \rangle \longrightarrow FBL(\mathbb{L})$ such that $\| u_{\varepsilon} \| \leq 1+\varepsilon$ and $R \circ u_{\varepsilon} = id_{FBL_*\langle \mathbb{L} \rangle}$.
\end{prop}

\section{Projectivity of the free Banach lattice generated by a finite lattice}

We are going to prove that if $\mathbb{L}$ is a finite lattice, then $FBL_*\langle \mathbb{L} \rangle$ is a projective Banach lattice.

\begin{prop}\label{LEstrellaANR}

If $\mathbb{L} = \set{0, \ldots, n-1}$ is a finite lattice, then $\mathbb{L}^* \cap \partial [-1,1]^n$ is $ANR$.

\end{prop}

\begin{proof}

Clearly, $\partial [-1,1]^n$ is a finite union of closed convex subsets of $\mathbb{R}^n$.

On the other hand, let $$A_{ijk} = \set{(x_1, \ldots, x_n) \in \mathbb{R}^n : x_i \vee x_j = x_k}$$ and $$B_{ijk} = \set{(x_1, \ldots, x_n) \in \mathbb{R}^n : x_i \wedge x_j = x_k}.$$

It is clear that 

$$A_{ijk} = \set{(x_1, \ldots, x_n) \in \mathbb{R}^n : x_i = x_k,\, x_j \leq x_i}
\cup \set{(x_1, \ldots, x_n) \in \mathbb{R}^n : x_j = x_k,\, x_i \leq x_j},$$

and

$$B_{ijk} = \set{(x_1, \ldots, x_n) \in \mathbb{R}^n : x_i = x_k,\, x_j \geq x_i}
\cup \set{(x_1, \ldots, x_n) \in \mathbb{R}^n : x_j = x_k,\, x_i \geq x_j}$$

are union of two closed convex sets.

Since $$\mathbb{L}^* = \left(\bigcap_{i \vee j = k} A_{ijk}\right) \bigcap \left(\bigcap_{i \wedge j = k} B_{ijk}\right),$$

we have that $\mathbb{L}^\ast $ is also a finite union of closed convex subsets of $\mathbb{R}^n$. We conclude that $\mathbb{L}^* \cap \partial [-1,1]^n$ is a finite union of closed convex subsets of $\mathbb{R}^n$ and thus ANR.

\end{proof}

In the context of compact metric spaces, the retractions in the definition of ANR can be taken arbitrarily close to the identity. We state this fact as a lemma in the particular case that we need:

\begin{lem}\label{closeretraction}

Let $\mathbb{L} = \set{0, \ldots, n-1}$ be a finite lattice. Then, given $\varepsilon > 0$ , there exist an open set $V_{\varepsilon} = V_{\varepsilon}(\mathbb{L}^*)$ with $\mathbb{L}^* \cap \partial [-1,1]^n \subset V_{\varepsilon} \subset \mathbb{R}^n$ and a continuous map $\varphi: V_{\varepsilon} \longrightarrow \mathbb{L}^* \cap \partial [-1,1]^n$ such that $\varphi \vert_{\mathbb{L}^* \cap \partial [-1,1]^n} = id_{\mathbb{L}^* \cap \partial [-1,1]^n}$ and $d(x^*, \varphi(x^*)) < \varepsilon$ for every $x^*  \in V_{\varepsilon}$, where $d$ is the square metric in $\mathbb{R}^n$.

\end{lem}

\begin{proof}

As $\mathbb{L}^* \cap \partial [-1,1]^n$ is an ANR by Proposition \ref{LEstrellaANR}, we cand find a bounded neighborhood $V$ of $\mathbb{L}^* \cap \partial [-1,1]^n$ in $\mathbb{R}^n$ and a retraction $\psi : V \longrightarrow \mathbb{L}^* \cap \partial [-1,1]^n$. Let us take an open set $W$ such that $\mathbb{L}^* \cap \partial [-1,1]^n \subset W \subset \overline{W} \subset V \subset \mathbb{R}^n$. Now, $\psi \vert_{\overline{W}} : \overline{W} \longrightarrow \mathbb{L}^* \cap \partial [-1,1]^n$ is a continuous map between compact metric spaces, so it is uniformly continuous. Given $\varepsilon > 0$, there exists $\delta > 0$ such that $d(\psi(x^*), \psi (y^*)) < \varepsilon/2$ if $x^*, y^* \in \overline{W}$ and $d(x^*,y^*) < \delta$. Put $\eta = \min (\varepsilon/2, \delta)$ and take $$V_{\varepsilon} = \set{x^* \in W : \text{there exists }y^* \in \mathbb{L}^* \cap \partial [-1,1]^n \text{ with }d(x^*,y^*) < \eta},$$
and $\varphi = \psi|_{ V_{\varepsilon}}: V_{\varepsilon} \longrightarrow \mathbb{L}^* \cap \partial [-1,1]^n$. Clearly, $\varphi$ is continuous and $\varphi \vert_{\mathbb{L}^* \cap \partial [-1,1]^n} = id_{\mathbb{L}^* \cap \partial [-1,1]^n}$. Let $x^* \in V_{\varepsilon}$, and let $y^* \in \mathbb{L}^* \cap \partial [-1,1]^n$ such that $d(x^*,y^*) < \eta$. Then, $$d(x^*, \varphi(x^*)) \leq d(x^*,y^*) + d(y^*, \varphi (x^*)) = d(x^*,y^*) + d(\varphi (y^*), \varphi (x^*)) < \frac{\varepsilon}{2} + \frac{\varepsilon}{2} = \varepsilon.$$
\end{proof}

\begin{thm}

If $\mathbb{L}$ is a finite lattice, then $FBL_* \langle \mathbb{L} \rangle$ is a projective Banach lattice.

\end{thm}

\begin{proof}

Let $n$ be the cardinality of $\mathbb{L}$. We may suppose that $\mathbb{L} = \{0,\ldots,n-1\}$ with some lattice operations, and in this way we identity $[-1,1]^\mathbb{L}$ with $[-1,1]^n$. We fix $\varepsilon > 0$, and we will construct the  map $u_{\varepsilon}:  FBL_*\langle \mathbb{L} \rangle \longrightarrow FBL(\mathbb{L})$ of Lemma~\ref{caractProyectividad}. Let $V_\varepsilon$ and $\varphi$ be given by Lemma~\ref{closeretraction}. By Urysohn's lemma, we can find a continuous function  $1_{\varepsilon}: \partial[-1,1]^n \longrightarrow [0,1]$ such that $1_{\varepsilon}(x^*) = 1$ if $x^* \in \mathbb{L}^* \cap \partial[-1,1]^n$, and $1_{\varepsilon}(x^*) = 0$ if $x^* \not\in V_{\varepsilon}$.  We define $u_{\varepsilon}(f)(x^*) = {1_{\varepsilon}}(x^*)\cdot f(\varphi(x^*))$ if $x^* \in V_{\varepsilon}$, and $u_{\varepsilon}(f)(x^*) = 0$ if $x^* \notin V_{\varepsilon}$, for every $f \in  FBL_*\langle \mathbb{L} \rangle$ and $x^* \in \partial[-1,1]^n$. We extend the definition for elements $x^\ast\in [-1,1]^n \setminus \partial [-1,1]^n$ in such a way that $u_\varepsilon(f)$ is positively homogeneous. Since $\mathbb{L}$ is finite, the fact that $u_\varepsilon(f)$ is continuous on $\partial [-1,1]^n$ and positively homogeneous guarantees that $u_\varepsilon(f)\in FBL(\mathbb{L})$. It is easy to check that $u_{\varepsilon}$ is a Banach lattice homomorphism and that $R \circ u_{\varepsilon} = id_{FBL_*\langle \mathbb{L} \rangle}$. It would remain to check that $\| u_{\varepsilon} \| \leq 1+\varepsilon$. We will prove instead that for this $u_\varepsilon$ we have $\|u_\varepsilon\|\leq 1 + n\varepsilon$, which is still good enough. We know that $$\| u_{\varepsilon} \| = \sup \set{\| u_{\varepsilon}(f) \| : f \in FBL_* \langle \mathbb{L} \rangle, \| f \|_* \leq 1},$$ where $$ \| u_{\varepsilon}(f) \| = \sup \set{\sum_{i = 1}^m \abs{r_i u_{\varepsilon}(f)(x_{i}^{\ast})} : x_i^{\ast} \in \partial[-1,1]^n, \text{ } r_i \in \mathbb{R},\text{ }\sup_{x \in \mathbb{L}} \sum_{i=1}^m \abs{r_i x_i^{\ast}(x)} \leq 1 }.$$

So we fix $f \in FBL_*\langle \mathbb{L} \rangle$ with $\norm{f}_* \leq 1$, where $$ \norm{f}_\ast = \sup \set{\sum_{i = 1}^m \abs{s_i f(y_{i}^{\ast})} : y_i^{\ast} \in \mathbb{L}^*, \text{ } s_i \in \mathbb{R},\text{ }\sup_{x \in \mathbb{L}} \sum_{i=1}^m \abs{s_i y_i^{\ast}(x)} \leq 1 },$$
and we want to prove that $\|u_\varepsilon(f)\|\leq 1+n\varepsilon$.
Using the expression of $\|u_\varepsilon(f)\|$ as a supremum, we pick $x_1^*, \ldots, x_m^* \in \partial[-1,1]^n$, $r_1, \ldots, r_m \in \mathbb{R}$ such that $\sup_{x \in \mathbb{L}} \sum_{i=1}^m \abs{r_i x_i^{\ast}(x)} \leq 1$, and we want to prove that 
$$\sum_{i = 1}^m \abs{r_i u_{\varepsilon}(f)(x_{i}^{\ast})}\leq 1 + n\varepsilon.$$

The first estimation is that $$\sum_{i = 1}^m \abs{r_i u_{\varepsilon}(f)(x_{i}^{\ast})} = \sum_{x_i^* \in V_{\varepsilon}} \abs{r_i  {1_{\varepsilon}}(x_i^*)f(\varphi(x_i^*))} \leq \sum_{x_i^* \in V_{\varepsilon}} \abs{r_i f(\varphi(x_i^*))}.$$ If we write ${y_i}^\ast := \varphi(x_i^*)$, the inequality above becomes
 $$(\star)\ \ \sum_{i = 1}^m \abs{r_i u_{\varepsilon}(f)(x_{i}^{\ast})} \leq \sum_{x_i^* \in V_{\varepsilon}} \abs{r_i f({y_i}^\ast)}.$$

On the other hand, if $x \in \mathbb{L}$ then

\begin{equation*} 
\begin{split} 
\sum_{x_i^* \in V_{\varepsilon}}\abs{r_i {y_i^*}(x)} &= \sum_{x_i^* \in V_{\varepsilon}}\abs{r_i \varphi(x_i^*)(x)}  \\ 
&\leq  \sum_{x_i^* \in V_{\varepsilon}}\abs{r_i x_i^*(x)} + \sum_{x_i^* \in V_{\varepsilon} }\abs{r_i} \abs{\varphi(x_i^*)(x) - x_i^*(x)}\\ 
&\leq 1 + \varepsilon  \sum_{x_i^* \in V_{\varepsilon} }\abs{r_i} \leq 1 + \varepsilon n.\\ 
\end{split} 
\end{equation*} 

The last inequality is because $x_i ^\ast\in \partial [-1,1]^n$, and therefore
\begin{equation*}
\sum_{i=1}^m|r_i| = \sum_{i=1}^m|r_i|\sup_{x\in\mathbb{L}}|x_i^\ast(x)| \leq \sum_{x\in\mathbb{L}}\sum_{i=1}^m |r_i| |x_i^\ast(x)| \leq |\mathbb{L}|\cdot 1 = n.
\end{equation*}

Taking $s_i = \frac{r_i}{1 + n\varepsilon}$, we have that, for all $x\in\mathbb{L}$, $$\sum_{x_i^* \in V_{\varepsilon}}\abs{s_i {y_i^*}(x)} = \sum_{x_i^* \in V_{\varepsilon}}\abs{\frac{r_i}{1 + n\varepsilon} {y_i^*}(x)} \leq 1.$$

Thus, the $s_i$ and the $y_i$ are as in the supremum that defines $\|f\|_\ast \leq 1$. Therefore $$\sum_{x_i^* \in V_{\varepsilon}}\abs{s_i f({y_i^*})} \leq 1,$$ and getting back to our initial estimation ($\star$), we get   $$\sum_{i = 1}^m \abs{r_i u_{\varepsilon}(f)(x_{i}^{\ast})} \leq \sum_{x_i^* \in V_{\varepsilon}}\abs{r_i f({y_i^*})} \leq 1 + n\varepsilon.$$ 

\end{proof}

\section{Projectivity of the free Banach lattice generated by an infinite linear order}

Now, we are going to prove that if $\mathbb{L}$ is an infinite linear order \textcolor{blue}{containing either an increasing sequence without upper bounds or a decreasing sequence without lower bounds}, then $FBL_*\langle \mathbb{L} \rangle$ is not projective. This will be a direct consequence of the fact that the free Banach lattice generated by the set of the natural numbers is not projective. In the proof, we will use the following:

\begin{lem}\label{funcionesenFBLN}

 Suppose that $\varphi_i: [-1,1]^{\mathbb{N}} \longrightarrow \mathbb{R}$, $i = 1,2, \ldots$, are continuous functions such that, for every $i$,
\begin{enumerate}
\item $\varphi_i((x_n)_{n \in \mathbb{N}}) = x_i$ whenever $x_1 \leq x_2 \leq \ldots$,
\item $\varphi_i(x) \leq \varphi_{i+1}(x)$ for all $x \in [-1,1]^{\mathbb{N}}$.
\end{enumerate}
Then, when we view the $\varphi_i$'s as elements of the free Banach lattice $FBL(\mathbb{N})$, the sequence of norms $\| \varphi_{i} \|_{FBL(\mathbb{N})}$ is unbounded. 

\end{lem}

\begin{proof}

Let $\pi_i : [-1,1]^{\mathbb{N}} \longrightarrow [-1,1]$ be the projection on the $i$-th coordinate. Consider the set $M := \set{(x_n)_{n \in \mathbb{N}} \in [-1,1]^{\mathbb{N}} : x_1 \leq x_2 \leq \ldots} \subset [-1,1]^{\mathbb{N}}$. Since $M$ is closed and $[-1,1]^{\mathbb{N}}$ with the product topology is compact, we have that $M$ is compact. Condition {(1)} in the lemma means that $\varphi_i\vert_M = \pi_i\vert_M$ for all $i$.
Using the compactness of $M$ and the continuity of $\varphi_i$ and $\pi_i$, this implies that there exists a neighborhood $U_{i}$ of $M$ such that $$d(\varphi_i\vert_{U_{i}}, \pi_i\vert_{U_{i}}) = \sup_{x \in U_{i}}| \varphi_i(x) - \pi_i(x) | < \frac{1}{2}.$$
For an integer $k \geq 3$, let $$W_{k} := \set{(x_n)_{n \in \mathbb{N}} \in [-1,1]^{\mathbb{N}} : x_i < x_j + k^{-1} \text{ whenever }i<j<k}.$$
The family $\set{W_{k} : k \geq 3}$ is a neighborhood basis of $M$. We define inductively an increasing sequence of natural numbers $k_0<k_1<k_2<k_3<\cdots$, and a sequence of points $y^1,y^2,\ldots\in [-1,1]^\mathbb{N}$ as follows. 
We take $k_0=1$ as a starting point of the induction. Suppose that we have defined $k_j$ for $j < n$. We choose $k_{j+1}>k_j$ such that $W_{k_{j+1}}\subset U_{k_j}$, and we define $y^{j+1} : \mathbb{N} \longrightarrow [-1,1]$ to be the map given by 
$$y^{j+1}(n)= \left\{ \begin{array}{lcc}
              0 &   if  & n < k_j, \\
              \\ 1 &  if  & k_j \leq n < k_{j+1}, \\
              \\ 0 & if & n \geq k_{j+1}.
              \end{array}
    \right.$$
We have $y^{j+1} \in W_{k_{j+1}}$, so  $| \varphi_{k_j}(y^{j+1}) - \pi_{k_j}(y^{j+1}) | = |\varphi_{k_j}(y^{j+1}) - 1 | < \frac{1}{2}$ and $\varphi_{k_j}(y^{j+1}) > \frac{1}{2}$.

When $j+1\leq m$, using condition (1) of the Lemma, we get that
$$\varphi_{k_m}(y^{j+1}) \geq \varphi_{k_{j}}(y^{j+1}) > \frac{1}{2}.$$

Remember how the norm is defined: $$\| \varphi \|_{FBL(\mathbb{N})} = \sup \set{ \sum_{j = 1}^m  |r_j \varphi (x_j)| : r_j \in \mathbb{R},\, x_j \in [-1,1]^{\mathbb{N}},\, \sup_{n \in \mathbb{N}} \sum_{j=1}^m |r_j x_j(n)| \leq 1 }.$$

We have that $\sup_{n \in \mathbb{N}} | y^1(n) + \cdots + y^m(n) | = 1$, therefore $$\| \varphi_{k_m} \|_{FBL(\mathbb{N})} \geq |\varphi_{k_m}(y^1)| + \cdots + |\varphi_{k_m}(y^m)| > \frac{m}{2}.$$
\end{proof}

\textcolor{red}{\textst{Now, let $\mathbb{N}^+ = \mathbb{N} \cup \lbrace +\infty \rbrace$.}}

\begin{lem}\label{Nnotproj} $FBL_* \langle \mathbb{N} \rangle$ \textcolor{red}{\textst{and $FBL_* \langle \mathbb{N}^+ \rangle$ are}} \textcolor{blue}{is} not projective.
\end{lem}

\begin{proof}

\textcolor{red}{\textst{First, if}} \textcolor{blue}{If} $FBL_* \langle \mathbb{N} \rangle$ was projective, then for $\varepsilon>0$ we would have a map  $u_{\varepsilon}:  FBL_*\langle \mathbb{N} \rangle \longrightarrow FBL(\mathbb{N})$ like in Proposition \ref{caractProyectividad}. Remember that if $i \in \mathbb{N}$, $\dot{\delta_i}: \mathbb{N}^* \longrightarrow \mathbb{R}$ is the map given by $\dot{\delta_i}(x^*) = x^*(i)$ for every $x^* \in \mathbb{N}^*$, that is an element of $FBL_\ast\langle\mathbb{N}\rangle$. We consider $\varphi_i = u_\varepsilon(\dot{\delta}_i)\in FBL(\mathbb{N})$, that we view as continuous functions $\varphi_i:[-1,1]^\mathbb{N}\To \mathbb{R}$. The fact that $u_\varepsilon$ is a lattice homomorphism gives condition (2) of Lemma \ref{funcionesenFBLN}, while the fact that $R\circ u_\varepsilon = id_{FBL_\ast\langle\mathbb{N}\rangle}$ gives condition (1) of Lemma \ref{funcionesenFBLN}. The fact that $\|u_\varepsilon\|\leq 1+\varepsilon$ contradicts the conclusion of Lemma \ref{funcionesenFBLN}.

\textcolor{red}{Here there was a proof that $FBL_* \langle \mathbb{N}^+ \rangle$ was not projective, but this was wrong.}

\end{proof}

The following fact can be viewed as a corollary of Proposition~\ref{quotientofprojective}, but we state if for convenience:

\begin{lem}\label{projcomplemented}
Let $P$ and $P'$ be Banach lattices, and let $\tilde{\pi}:P\To P'$ and $\tilde{\imath}:P'\To P$ be Banach lattice homomorphisms such that $\|\tilde{\imath}\| = \|\tilde{\pi}\|=1$ and $\tilde{\pi}\circ\tilde{\imath} = id_{P'}$. If $P$ is projective, then $P'$ is projective.
\end{lem}

\begin{proof}
In order to check the projectivity of $P'$, let $Q:X\To X/J$, $T':P'\To X/J$ and $\varepsilon > 0$ be as in Definition~\ref{projdef}. Then we can apply the projectivity of $P$ considering $T=T'\circ\tilde{\pi}$, so we get $\hat{T}:P\To X$ such that $Q\circ \hat{T} = T'\circ\tilde{\pi}$ and $\|\hat{T}\|\leq (1 + \varepsilon) \|T\|\leq (1 + \varepsilon) \|T'\|$. The desired lift is $\hat{T}' = \hat{T} \circ \tilde{\imath}$. On the one hand $\|\hat{T}'\| \leq \|\hat{T}\| \leq (1 + \varepsilon) \|T'\|$, and on the other hand $Q\circ \hat{T} \circ \tilde{\imath} = T'\circ \tilde{\pi} \circ \tilde{\imath} = T'$.
\end{proof}

\begin{thm}
Let $\mathbb{L}$ be an infinite linearly ordered set \textcolor{blue}{containing either an increasing sequence without upper bounds or a decreasing sequence without lower bounds}. Then, $FBL_* \langle \mathbb{L} \rangle$ is not projective.
\end{thm}

\begin{proof}

\textcolor{red}{\sout{$\mathbb{L}$ contains either an increasing or a decreasing sequence.}} Let us suppose without loss of generality that it contains an increasing  sequence $x_1<x_2<x_3<\cdots$ \textcolor{blue}{without upper bounds}.

\textcolor{red}{\sout{First, suppose that it has no upper bound.}} The map $\imath : (\mathbb{N}, \leq) \longrightarrow (\mathbb{L}, \preceq)$ given by $\imath(n) = x_n$ for every $n \in \mathbb{N}$ is a lattice homomorphism. Let $\pi : \mathbb{L}  \longrightarrow \mathbb{N}$ be the map given by 

$$\pi(x)= \left\{ \begin{array}{lcc}
              1 & \text{if} & x < x_2, \\
              n &  \text{if}  & x \in [x_n, x_{n+1}) \text{ for any } n \geq 2 .
              \end{array}
    \right.$$
    
Notice that $\pi$ is also a lattice homomorphism and $\pi\circ \imath = id_{\mathbb{N}}$. We are going to use the universal property of the free Banach lattice over a lattice as stated in Definition~\ref{FBLGBL}. Let $\phi_{\mathbb{L}}$ and $\phi_{\mathbb{N}}$ be the canonical inclusion of $\mathbb{L}$ and $\mathbb{N}$ into $FBL_*\langle \mathbb{L} \rangle$ and $FBL_*\langle \mathbb{N} \rangle$, respectively, and let $\tilde{\imath} : FBL_* \langle \mathbb{N} \rangle \longrightarrow FBL_* \langle \mathbb{L} \rangle$ and $\tilde{\pi} : FBL_*\langle \mathbb{L} \rangle \longrightarrow FBL_* \langle \mathbb{N} \rangle$ be the corresponding extensions of $\phi_{\mathbb{L}} \circ \imath$ and $\phi_{\mathbb{N}} \circ \pi$ according to Definition~\ref{FBLGBL}. The composition $\tilde{\pi}\circ\tilde{\imath}$ and the identity mapping $FBL_*\langle \mathbb{N} \rangle \longrightarrow FBL_* \langle \mathbb{N} \rangle$ are both extensions of $\phi_{\mathbb{N}}$ so by the uniqueness in Definition~\ref{FBLGBL}, $\tilde{\pi}\circ\tilde{\imath} = id_{FBL_*\langle\mathbb{N}\rangle}$. We can apply Lemma~\ref{projcomplemented}, so if $ FBL_*\langle \mathbb{L} \rangle$ was projective, then $FBL_*\langle \mathbb{N} \rangle$ would also be projective, in contradiction with Lemma~\ref{Nnotproj}.

\textcolor{red}{The rest of the proof must be omitted}

%
\end{proof}

\end{document}